\documentclass[11pt, a4paper]{article}

\usepackage{graphicx,tikz,caption,subcaption}
\usepackage{fullpage}
\usepackage{hyperref}
\usepackage{enumitem,verbatim}
\usepackage{hyperref}
\usepackage{amsmath}
\usepackage{amssymb}
\usepackage{amsfonts}
\usepackage{amsthm}
\usepackage[capitalise]{cleveref}
\usepackage{bm}

\usepackage[margin=2.5cm]{geometry}

\newtheorem{theorem}{Theorem}[section]
\newtheorem{lemma}[theorem]{Lemma}
\newtheorem{claim}[theorem]{Claim}
\newtheorem{cor}[theorem]{Corollary}

\newtheorem{conjecture}[theorem]{Conjecture}

\newtheorem*{observation*}{Observation}
\newtheorem{proposition}[theorem]{Proposition}
\newtheorem{problem}[theorem]{Problem}
\newtheorem{question}[theorem]{Question}
\newtheorem*{question*}{Question}
\newtheorem{innerdefinition}{Definition}
\newenvironment{definition*}
  {
   \innerdefinition}
  {\endinnerdefinition}

\theoremstyle{definition}
\newtheorem{defn}[theorem]{Definition}

\theoremstyle{remark}

\newcounter{propcounter}

\newenvironment{poc}{\begin{proof}[Proof of claim]}{\end{proof}}

\usepackage{todonotes}

\newcommand{\ex}{\mathrm{ex}}

\usepackage{tikz}
\tikzstyle{aNode} = [circle, fill = black]
\tikzstyle{bNode} = [circle,draw = black, thick]

\newcommand{\ppoints}[1]{%
\begin{tikzpicture}[inner sep = 0.7pt, #1]%
\node (1) at (0,-2) [aNode]{};
\node (3) at (1.5,-2) [aNode]{};
\node (2) at (0.75,-1) [aNode]{};
\end{tikzpicture}%
}

\def\points{\ppoints{scale=0.08}}
\def\co{\mathrm{co}}

\usetikzlibrary{patterns}
\usetikzlibrary{shapes}
\usetikzlibrary{decorations.pathreplacing}
\usetikzlibrary{decorations.pathmorphing}
\usetikzlibrary{positioning}

\title{On $3$-graphs with vanishing codegree Tur\'{a}n density}
\author{Laihao Ding\thanks{School of Mathematics and Statistics, and Key Laboratory of Nonlinear Analysis \& Applications (Ministry of Education), Central China Normal University, Wuhan 430079, China and Extremal Combinatorics and Probability Group (ECOPRO),  Institute for Basic Science (IBS), Daejeon, South Korea. Supported by the National Nature Science Foundation of China (11901226), the China Scholarship Council and IBS-R029-C4. \textbf{Email address}: dinglaihao@ccnu.edu.cn}
\and Ander Lamaison\thanks{Extremal Combinatorics and Probability Group (ECOPRO),  Institute for Basic Science (IBS), Daejeon, South Korea. Supported by IBS-R029-C4. \textbf{Email address}: \{ander,hongliu\}@ibs.re.kr}
\and Hong Liu\footnotemark[2]
\and Shuaichao Wang\thanks{Center for Combinatorics and LPMC, Nankai University, Tianjin 300071, China and Extremal Combinatorics and Probability Group (ECOPRO),  Institute for Basic Science (IBS), Daejeon, South Korea. Supported by the China Scholarship Council and IBS-R029-C4. \textbf{Email address}: wsc17746316863@163.com } 
\and Haotian Yang\thanks{Taishan College, Shandong University, Jinan 250100, China and Extremal Combinatorics and Probability Group (ECOPRO),  Institute for Basic Science (IBS), Daejeon, South Korea. Supported by Seed Fund Program for International Research Cooperation of Shandong University and IBS-R029-C4. \textbf{Email address}: 202017091012@mail.sdu.edu.cn }}

\begin{document}
\maketitle

\begin{abstract}
 For a $k$-uniform hypergraph (or simply $k$-graph) $F$, the codegree Tur\'{a}n density $\pi_{\co}(F)$ is the supremum over all $\alpha$ such that there exist arbitrarily large $n$-vertex $F$-free $k$-graphs $H$ in which every $(k-1)$-subset of $V(H)$ is contained in at least $\alpha n$ edges. Recently, it was proved that for every $3$-graph $F$, $\pi_{\mathrm{co}}(F)=0$ implies $\pi_{\points}(F)=0$, where $\pi_{\points}(F)$ is the uniform Tur\'{a}n density of $F$ and is defined as the supremum over all $d$ such that there are infinitely many $F$-free $k$-graphs $H$ satisfying that any induced linear-size subhypergraph of $H$ has edge density at least $d$. 

  In this paper, we introduce a layered structure for $3$-graphs which allows us to obtain the reverse implication: every layered $3$-graph $F$ with $\pi_{\points}(F)=0$ satisfies $\pi_{\mathrm{co}}(F)=0$. Along the way, we answer in the negative a question of Falgas-Ravry, Pikhurko, Vaughan and Volec [J. London Math. Soc., 2023] about whether $\pi_{\points}(F)\leq\pi_{\mathrm{co}}(F)$ always holds. In particular, we construct counterexamples $F$ with positive but arbitrarily small $\pi_{\mathrm{co}}(F)$ while having $\pi_{\points}(F)\ge 4/27$.

\end{abstract}

\section{Introduction}
Given a $k$-uniform hypergraph $F$ (or simply \emph{$k$-graph}), the \emph{Tur\'{a}n number} of $F$, denoted by ${\rm ex}(n,F)$, is the maximum number of edges in an $n$-vertex $k$-graph $H$ containing no copy of $F$. Within the field of extremal combinatorics, Tur\'an-type problems represent one of the most important topics of study, dating back to the theorems of Mantel and Tur\'an in the early 20th century. In the decades since, Tur\'an-type problems have found applications and numerous connections in other fields, ranging from error-correcting codes in information theory to additive number theory to sphere packing, just to name a few. Targeting on the limit behavior, one may define the \emph{Tur\'{a}n density} 
$$\pi(F):=\lim\limits_{n\to \infty}\frac{\ex(n,F)}{\binom{n}{k}}.$$
A standard averaging argument shows that this limit always exists.
While Tur\'{a}n densities are well-understood for graphs (i.e., $2$-graphs), determining the Tur\'{a}n density of a $k$-graph becomes notoriously difficult when $k\geq 3$. Despite much effort and countless attempts, even the Tur\'{a}n densities of the $3$-graphs on four vertices with three and four edges, denoted by $K_4^{(3)-}$ and $K_4^{(3)}$ respectively, are still unknown. Determining the value of $\pi(K_4^{(3)})$ is one of the major open problems in extremal combinatorics, with Erd\H{o}s~\cite{erdosreward} offering a \$500 reward for a solution. 

The Tur\'{a}n density $\pi(F)$ can also be viewed asymptotically as the largest (normalized) minimum degree of an $F$-free $k$-graph, i.e.~the supremum over all $d$ such that there are arbitrarily large $n$-vertex $F$-free $k$-graphs $H$ in which every vertex is contained in at least $d{n \choose k-1}$ edges in $H$. A natural variation of the Tur\'{a}n density, introduced by Mubayi and Zhao \cite{mubayi2007co}, is the codegree Tur\'{a}n density defined as follows. Given a $k$-graph $H$, the \emph{degree} $d_{H}(S)$ of a set of vertices $S$ is the number of edges containing it. The \emph{minimum codegree} $\delta_{\co}(H)$ of $H$ is the minimum of $d_{H}(S)$ over all $(k-1)$-subsets $S$ of $V(H)$. The \textit{codegree Tur\'{a}n number} $\ex_{\co}(n,F)$ is the maximum $\delta_{\co}(H)$ an $n$-vertex $F$-free $k$-graph $H$ can admit, and the \textit{codegree Tur\'{a}n density} $\pi_{\co}(F)$ is defined as 
$$\pi_{\co}(F)=\lim\limits_{n\to \infty}\frac{\ex_{\co}(n,F)}{n}.$$ 
This limit always exists (\cite{mubayi2007co}), and it is not hard to see that $\pi_{\co}(F)\leq\pi(F)$ for any $k$-graph $F$. In particular, $\pi_{\co}(F)=\pi(F)$ when $F$ is a graph.

For $k\geq3$, similar as the Tur\'{a}n density, determining the codegree Tur\'{a}n density of a $k$-graph seems also very difficult in general. In the late 1990s, Nagle \cite{K43-nagle} and Czygrinow and Nagle \cite{K43nagle} conjectured that $\pi_{\co}(K_4^{(3)-})=\frac{1}{4}$ and $\pi_{\co}(K_4^{(3)})=\frac{1}{2}$, respectively. The $\pi_{\co}(K_4^{(3)-})$ case was only recently settled by Falgas-Ravry, Pikhurko, Vaughan and Volec \cite{K43-codegree} via the flag algebra technique. There are few sporadic 3-graphs whose codegree Tur\'an densities are known: the Fano plane~\cite{mubayi2005co}, $F_{3,2}$ with $V(F_{3,2})=[5]$ and $E(F_{3,2})=\{123,124,125,345\}$~\cite{falgas2015codegree}, and $C_{\ell}^{(3)-}$ with $\ell\geq5$, the tight cycle of length $\ell$ with one edge removed~\cite{piga2022codegree}.

Recently, Piga and Sch\"{u}lke \cite{piga2023hypergraphs} showed that surprisingly the codegree Tur\'{a}n density can be arbitrarily close to zero for $k$-graphs when $k\geq3$. Among all known variations of Tur\'{a}n density \cite{l2norm,positive,l-degree,ldegreeposi,reiher2018hypergraphs}, this is the first example with zero being an accumulation point. For instance, there is no $k$-graph with its Tur\'{a}n density and positive codegree Tur\'{a}n density lying in $(0,k!/k^k)$ and $(0,1/k)$ respectively, and no $3$-graph with its uniform Tur\'{a}n density lying in $(0,1/27)$. So it would be very interesting if one can characterize all $k$-graphs with zero codegree Tur\'{a}n density. 

\begin{problem}\label{chac}
 For $k\geq 3$, characterize all $k$-graphs $F$ with $\pi_{\co}(F)=0$.
\end{problem}

Based on the result of Piga and Sch\"{u}lke \cite{piga2023hypergraphs}, \cref{chac} is likely to be very challenging as one cannot directly mimic the known characterization of the usual Tur\'{a}n density or other variations being zero by  a \emph{single} lower bound construction to avoid all $k$-graphs with positive codegree Tur\'{a}n densities.


\subsection{A necessary condition} 

For real numbers $d\in[0,1]$ and $\eta>0$, an $n$-vertex $k$-graph $H$ is said to be \emph{uniformly $(d,\eta)$-dense} if for all $U\subseteq V(H)$, it holds that $\left| \binom{U}{k}\cap E(H)\right|\geq d\binom{|U|}{k}-\eta n^k.$
Given a $k$-graph $F$, the \emph{uniform Tur\'{a}n density} of $F$ is defined as 
\begin{align*}\label{dot-turan-dense}
\pi_{\points}(F) = \sup \{ d\in [0,1] & : \text{for\ every\ } \eta>0 \ \text{and\ } n_0\in \mathbb{N},\ \text{there\ exists\ an\ } F \text{-free}, \\
&\quad \text{uniformly}\ (d,\eta) \text{-dense}\ k \text{-graph\ } H\ \text{with\ } |V(H)|\geq n_0 \}.
\end{align*}
Erd\H{o}s and S\'{o}s \cite{erdossos} first considered the uniform Tur\'{a}n problems for $3$-graphs and conjectured that $\pi_{\points}(K_4^{(3)-})=\frac{1}{4}$. This conjecture was recently confirmed by Glebov, Kr\'{a}l' and Volec \cite{K43-dan} using flag algebra, and later by Reiher, R\"{o}dl and Schacht \cite{K43-rodl} via the hypergraph regularity method. For $K_4^{(3)}$, a construction due to R\"{o}dl \cite{K43rodl} shows that $\pi_{\points}(K_4^{(3)})\geq\frac{1}{2}$, but whether $\frac{1}{2}$ is the correct value of $\pi_{\points}(K_4^{(3)})$ still remains open. Since then, the hypergraph regularity method has been widely used in this area and many results have been obtained \cite{cycledan,Bjarne,1/27dan,zhou1,reiher2018hypergraphs}. In particular, Reiher, R\"{o}dl and Schacht \cite{reiher2018hypergraphs} characterized all $3$-graphs with vanishing uniform Tur\'{a}n density, and provided a construction showing that the uniform Tur\'{a}n density cannot lie in the interval $(0,\frac{1}{27})$. 
For more results and problems on this topic and other variants, we refer the readers to~\cite{orderfive,ding1,ding2,lenz,zhou2,reihersurvey,K43dotedge,manteluniform}.

In \cite{unico}, Ding, Liu, Wang and Yang gave a necessary condition for a $3$-graph having vanishing codegree Tur\'{a}n density using its uniform Tur\'{a}n density.

\begin{theorem}[Ding, Liu, Wang and Yang, \cite{unico}]\label{main}
    Let $F$ be a $3$-graph. If $\pi_{\mathrm{co}}(F)=0$, then $\pi_{\points}(F)=0$.
\end{theorem}

The following closely related question was recently raised by Falgas-Ravry, Pikhurko, Vaughan and Volec \cite{K43-codegree}.

\begin{question}[\cite{K43-codegree}]\label{ques}
For any $3$-graph $F$, is it true that $\pi_{\points}(F)\leq\pi_{\mathrm{co}}(F) ?$
\end{question}

Falgas-Ravry and Lo \cite{loanswer} provided a positive answer to this question under a stronger uniform denseness assumption. We answer~\cref{ques} in the negative by providing an infinite sequence of counterexamples.

\begin{theorem}\label{counterex}
For every $\varepsilon>0$, there exists a $3$-graph $F$ with $\pi_{\points}(F)\geq 4/27$ and $\pi_{\mathrm{co}}(F)\leq\varepsilon$.
\end{theorem}

Note that Theorem \ref{main} together with \cref{counterex} provides an alternative proof of the result of Piga and Sch\"{u}lke \cite{piga2023hypergraphs} that the codegree Tur\'{a}n density can be arbitrarily close to zero for $3$-graphs.

\subsection{A sufficient condition} 

At the moment, there are very few known non-trivial (i.e. not tripartite) examples of $3$-graphs with zero codegree Tur\'an density. These are limited to tight cycles of length $\ell\geq5$ with one edge removed, $C_\ell^{(3)-}$, proved by by Piga, Sales, and Sch\"{u}lke \cite{piga2022codegree}, and zycles of length $\ell\geq 3$ with one edge removed, $Z_\ell^{(3)-}$, proved by Piga and Sch\"{u}lke \cite{piga2023hypergraphs}. Here, the \emph{zycle} of length $\ell$ is defined as the $3$-graph $F$ with $V\left(F\right)=\cup_{i=1}^{\ell}\{v_i,u_i\}$ and $E\left(F\right)=\left(\cup_{i=1}^{\ell-1}\{u_iv_iu_{i+1},u_iv_iv_{i+1}\}\right)\cup\{u_{\ell}v_{\ell}u_{1},u_{\ell}v_{\ell}v_{1}\}$.

 In this paper, we prove that $\pi_{\co}(F)=0$ is equivalent to $\pi_{\points}(F)=0$ for a large class of $3$-graphs $F$, which we call layered $3$-graphs and are defined as follows with a hierarchical structure and some kind of ``$1$-degenerateness''. In particular, our result generalizes those in~\cite{piga2022codegree,piga2023hypergraphs}: the above two known examples of 3-graphs with zero codegree Tur\'an density are layered 3-graphs.

A $3$-graph $F$ is called \textit{layered} if there exists a function $f:V(F)\rightarrow \mathbb{N}$ such that the following conditions hold.

\stepcounter{propcounter}
\begin{enumerate}[label = \rm({\bfseries \Alph{propcounter}\arabic{enumi}})]            \item\label{uniquemax} In each edge $uvw$ there is one vertex whose label is strictly greater than the other two.
    \item\label{fixmax} If two edges $uvw$ and $u'v'w'$ satisfy $\max\{f(u),f(v),f(w)\}=\max\{f(u'), f(v'), f(w')\}$, then the labels in both edges are a permutation of each other.
    \item\label{fixtwo} If two edges $uvw$ and $u'v'w'$ satisfy $f(u)=f(u')$ and $f(v)=f(v')$, then $f(w)=f(w')$.
\end{enumerate}

\noindent We call $f$ a \textit{layered function} of $F$, and call the set of vertices $v\in V(F)$ such that $f(v)$ is the $i$-th smallest integer in the range of $f$ the \textit{$i$-th layer} of $F$. 

\smallskip

\begin{theorem}\label{layered}
    If $F$ is a layered $3$-graph, then $\pi_{\mathrm{co}}(F)=0$ if and only if $\pi_{\points}(F)=0$.
\end{theorem}

\cref{layered} gives a sufficient condition for a 3-graph to have vanishing codegree Tur\'an density: if a $3$-graph $F$ is layered and satisfies $\pi_{\points}(F)=0$, then $\pi_{\co}(F)=0$. While $\pi_{\points}(F)=0$ is indispensable by Theorem \ref{main}, we believe that the layered structure is also necessary. We propose the following conjecture characterizing $3$-graphs with vanishing codegree Tur\'an density. 

\begin{conjecture}\label{Conj1}
    For a $3$-graph $F$, $\pi_{\co}(F)=0$ if and only if $F$ is layered and satisfies $\pi_{\points}(F)=0$.
\end{conjecture}

It is known that $\pi_{\points}(F)=0$ holds for any linear $3$-graph $F$~(see \cref{rodl}). 
So a special case of Conjecture \ref{Conj1} is the following.

\begin{conjecture}\label{Conj2}
    For a linear $3$-graph $F$,  $\pi_{\mathrm{co}}(F)=0$ if and only if $F$ is layered. 
\end{conjecture}

We show that this seemingly weaker conjecture is in fact equivalent to~\cref{Conj1}, which suggests that the linear $3$-graph case might be crucial for resolving Conjecture \ref{Conj1}.

\begin{theorem}\label{thm:C1-C2-same}
      \cref{Conj1} and \cref{Conj2} are equivalent.
\end{theorem}

We now present some applications of~\cref{layered}. The first one characterizes $3$-graphs in a special family with vanishing codegree Tur\'an density. A layered 3-graph $F$ on two layers is said to be a \emph{$(2,1)$-type $3$-graph}, that is, $V(F)$ can be partitioned into two parts such that each edge of $F$ has two vertices in one part (i.e. the first layer) and one vertex in the other part (i.e. the second layer). As the simplest layered $3$-graphs, $(2,1)$-type $3$-graphs play a pivotal role in both of our main results~\cref{main} and~\cref{layered}. We need some definitions. Given a vertex $u$ in a 3-graph $F$, its \emph{link graph}, denoted by $L_F(u)$, is the graph with $V(L_F(u))=V(F)\setminus\{u\}$ and $E(L_F(u))=\{vw : uvw\in E(F)\}$. Let $S$ be a finite set, we say that $\sigma$ is a \emph{labeling} of $S$ if $\sigma:S\rightarrow [|S|]$ is a bijection. Let $G$ be a graph and $\sigma$ be a labeling of $V(G)$. Let $u,v,w\in V(G)$ and $uvw$ form a path of length two in $G$. We say that $uvw$ is a \emph{monotone $P_3$} if $\sigma(u)<\sigma(v)<\sigma(w)$ or $\sigma(u)>\sigma(v)>\sigma(w)$.

It is well known that for a $3$-graph $F$, $\pi(F)=0$ if and only if $F$ is tripartite. Note that a tripartite $3$-graph is of $(2,1)$-type. Together with \cref{strengthen}, \cref{layered} implies the following characterization of $(2,1)$-type $3$-graphs $F$ with $\pi_{\co}(F)=0$.

\begin{cor}\label{twoone}
For a $(2,1)$-type $3$-graph $F$, $\pi_{\co}(F)=0$ if and only if there is a labeling of the first layer of $F$ such that $L_F(v)$ contains no monotone $P_3$ for every vertex $v$ in the second layer.
\end{cor}

Note that if a graph has no monotone $P_3$ under a labeling, then the graph must be bipartite. So for a $(2,1)$-type $3$-graph $F$ with $\pi_{\co}(F)=0$, the link graphs
$L_F(v)$ are bipartite for all vertices $v$ in the second layer (the bipartitions could be different). Thus,~\cref{twoone} implies that every linear $(2,1)$-type $3$-graph $F$, e.g. Fano plane with one edge removed, satisfies $\pi_{\co}(F)=0$. 

In the next application, we use~\cref{layered} and~\cref{twoone} to  recover the results for $C_{\ell}^{(3)-}$ and $Z^{(3)-}_{r}$ in~\cite{piga2022codegree,piga2023hypergraphs}.

\begin{cor}
For $\ell\geq 5$ and $r\geq 3$, $\pi_{\co}(C_{\ell}^{(3)-})=\pi_{\co}(Z^{(3)-}_{r})=0$.
\end{cor}

\begin{proof}
As indicated in \cite{piga2022codegree}, every $C_{\ell}^{(3-)}$ with $\ell\geq 5$ is contained in a blow-up of $C_5^{(3)-}$, so we only need to show that $\pi_{\co}(C_{5}^{(3)-})=0$ as any hypergraph and its blow-up have the same codegree Tur\'{a}n density (see e.g.~\cite{mubayi2007co}). Now one can easily check that $C_5^{(3)-}$ is a $(2,1)$-type $3$-graph satisfying the condition in \cref{twoone}.

For $Z^{(3)-}_{r}$ with $r\geq 3$, suppose  
 $$V\left(Z^{(3)-}_{r}\right)=\cup_{i=1}^{r}\{v_i,u_i\} \text{~and~} E\left(Z^{(3)-}_{r}\right)=\left(\cup_{i=1}^{r-1}\{u_iv_iu_{i+1},u_iv_iv_{i+1}\}\right)\cup\{u_{r}v_{r}u_{1}\}.$$ 
Let $\sigma(u_i)=2i-1$ and $\sigma(v_i)=2i$ for $1\leq i\leq r$,
and define
 \begin{equation}
 f(u)=
\begin{cases}
r,    &  \text{if $u=u_1$;}\\
r+1,          &  \text{if $u=v_1$;}\\
i-1,    &  \text{if $u\in\{u_i,v_i\}$ and $2\leq i\leq r$.}\nonumber
\end{cases}   
\end{equation}
It is not hard to verify that $\sigma$ is a labeling satisfying \ref{coloring} in \cref{rodl} and $f$ is a layered function. Therefore, $\pi_{\co}(Z^{(3)-}_{r})=0$ follows from \cref{layered}.
\end{proof}

\noindent{\bf Our approach.~}
To prove~\cref{counterex}, we utilize the so-called tensor product and the fact that the product is contained in large blow-ups of any component. We then observe that for any $3$-graph $F$ with minimum codegree at least two, there is a supersaturation phenomenon for $F$ in $3$-graphs with linear minimum codegree. On the other hand, such $F$ has positive uniform Tur\'an density. We can then take tensor product of such $F$ to obtain counterexamples.

To prove Theorem \ref{layered}, we first note a natural connection between certain half-bipartite graphs and the vanishing condition in \cref{permutation}, as well as a connection between graph distributions and linear codegree condition. Then applications of Ramsey's theorem show that in dense graph distributions any half-bipartite graph on a fixed vertex set will appear with a positive probability (Lemmas \ref{halfbip} and \ref{hb}). This enables us to embed any layered $3$-graph on two layers and with vanishing uniform Tur\'{a}n density into a $(2,1)$-type 3-graph, assuming that any pair in one part of the $(2,1)$-type 3-graph has a positive degree in the other part (Lemma \ref{add21}), which is one of the key steps in our proof. Since the $(2,1)$-type vanishing condition (\cref{strengthen}) is well compatible with the layered structure (Lemmas \ref{linked} and \ref{graphsum}), gluing in a correct way several copies of the $3$-graph obtained by removing from $F$ all vertices on the highest layer, we can inductively embed layer by layer any layered $3$-graph $F$ with $\pi_{\points}(F)=0$ into any $3$-graph with a positive codegree condition using \cref{addtripartite} and \cref{add21}.

\medskip

\noindent{\bf Notations.~}Let $F$ be a 3-graph. 
The \emph{shadow graph} of $F$, denoted by $\partial F$, is the graph with $V(\partial F)=V(F)$ and $E(\partial F)=\{uw : uvw\in E(F)\}$. Let $u,v$ be two vertices of the $3$-graph $F$. We say that $w\in V(F)$ is a \emph{coneighbor} of $u,v$ if $uvw\in E(F)$. The \emph{coneighbor set} of $u,v$ is defined as $N_{F}(uv)=\{w : uvw\in E(F)\}$. For a (hyper)graph $G$ and a set of vertices $W\subseteq V(G)$, denote by $G-W$ the (hyper)graph obtained from deleting vertices in $W$ from $G$.

The following operation provides us a natural way to merge several labelings into a larger one. Let $S_1,S_2$ be two disjoint finite sets and let $\sigma_1,\sigma_2$ be two labelings of $S_1,S_2$ respectively. The sum of $\sigma_1$ and $\sigma_2$, denoted by $\sigma_1 \oplus \sigma_2$, is a labeling of $S_1\cup S_2$ where
\[\sigma_1 \oplus \sigma_2(s)= \left \{
\begin{array}{ll}
 \sigma_1(s),  & s\in S_1;\\
 \sigma_2(s)+|S_1|, & s\in S_2.
\end{array}
\right.\]
For more than two labelings $\sigma_1,\sigma_2,\ldots,\sigma_k$, the sum of them, denoted by $\sum_{i=1}^{k}\sigma_i$, is inductively defined by 
\[\sum_{i=1}^{k}\sigma_i=\left(\sum_{i=1}^{k-1}\sigma_i\right)\oplus \sigma_k.\]
Let $S$ be a finite set and $\sigma':S\rightarrow \mathbb{Z}$ be an injection. We say that a labeling $\sigma$ of $S$ is \emph{induced by} $\sigma'$ if for every $s_1,s_2\in S$, $\sigma(s_1)>\sigma(s_2)$ if and only if $\sigma'(s_1)>\sigma'(s_2)$. Obviously, $\sigma$ exists and is unique.

\section{A counterexample to Question~\ref{ques}}

In this section, we prove Theorem~\ref{counterex} to show the existence of $3$-graphs $F$ with $\pi_{\points}(F)>\pi_{\mathrm{co}}(F)$. We start by constructing the $3$-graph $F$, then show that $\pi_{\mathrm{co}}(F)\leq\varepsilon$, and finally that $\pi_{\points}(F)\geq 4/27$.

\begin{proof}[Proof of Theorem~\ref{counterex}]

For each integer $k\geq 4$, let $\mathcal{F}_k$ be the family of all $k$-vertex $3$-graphs with minimum codegree at least two. Denote the elements of $\mathcal{F}_k$ as $\{F_1, F_2, \dots, F_\ell\}$. We define the $3$-graph $\widetilde{F}_k$ with vertex set $V(F_1)\times V(F_2)\times\dots\times V(F_\ell)$, and a triple of vertices $u,v,w\in V(F)$ form an edge if for every $1\leq i\leq \ell$, the $i$-th coordinates of the three vertices form an edge of $F_i$. This is sometimes referred to as the \textit{tensor product} of the elements of $\mathcal{F}_k$. A simple but useful fact about $\widetilde{F}_k$ we shall use is that $\widetilde{F}_k$ is contained as a subgraph in a $k^{{\ell-1}}$-blowup of any $F_i\in\mathcal{F}_k$ where each vertex is replaced by $k^{{\ell-1}}$ copies.  

\begin{claim}
    For every $\varepsilon>0$, there exists $k$ such that $\pi_{\mathrm{co}}(\widetilde{F}_k)\leq\varepsilon$.
\end{claim}
\begin{poc}
We choose parameters satisfying $1/n\ll 1/k \ll \varepsilon$. Let $H$ be an $n$-vertex $3$-graph with $\delta_{\mathrm{co}}(H)\geq\varepsilon n$. Let $S$ be a set of $k$ vertices from $H$, sampled uniformly at random. Given a pair of vertices $u,v\in S$, the probability that their codegree is at most 1 can be upper-bounded by $(k-2)(1-\varepsilon)^{k-3}$, since at least $k-3$ of the remaining vertices of $S$ must be excluded from $N_H(uv)$, and there are $k-2$ ways of selecting the potential neighbor. Therefore, the probability that $\delta_{\mathrm{co}}(H[S])\leq 1$ is at most ${k \choose 2}(k-2)(1-\varepsilon)^{k-3}$, which is less than $1/2$ for $k$ large enough.

Therefore, there are ${n \choose k}/2$ ways of selecting $S$ so that $H[S]$ is an element of $\mathcal{F}_k$. So by the pigeonhole principle there exists $i$ such that there are at least ${n \choose k}/(2|\mathcal{F}_k|)$ choices of $S$ such that $H[S]=F_i$. For $n$ large enough, $H$ contains a $k^{{\ell-1}}$-blowup of $F_i$ which contains $\widetilde{F}_k$ as a subgraph, so we conclude that $\pi_{\mathrm{co}}(\widetilde{F}_k)\leq\varepsilon$.  
\end{poc}

\begin{claim}\label{427}
    For each $k\geq 4$, we have $\pi_{\points}(\widetilde{F}_k)\geq 4/27$.
\end{claim}
\begin{poc}
Let $G$ be a complete ($2$-)graph on the vertex set $[n]$, where the edges are randomly colored red with probability $2/3$ and blue with probability $1/3$. Then construct the $3$-graph $H$ on the same vertex set $[n]$ by placing an edge on the triples $r<s<t$ if $rs$ and $rt$ are red and $st$ is blue. By a standard probabilistic argument, one can show that for any $\eta>0$, with high probability, $H$ is uniformly $(4/27, \eta)$-dense. We will show that $H$ does not contain a copy of $\widetilde{F}_k$.

Suppose on the contrary that $\widetilde{F}_k\subseteq H$. Then by the construction of $H$, there is a labeling $\sigma$ of $V(\widetilde{F}_k)$ and an edge coloring of $\partial\widetilde{F}_k$ such that for any $uvw\in E(\widetilde{F}_k)$ with $\sigma(u)<\sigma(v)<\sigma(w)$, $uv$ and $uw$ are red and $vw$ is blue. We say that two vertices $u, v\in V(\widetilde{F}_k)$ are \textit{disjoint} if for all $1\leq i\leq \ell$, the $i$-th coordinates of $u$ and $v$ are distinct. Note that by the construction of $\widetilde{F}_k$, a pair of vertices are contained in at least two edges if and only if they are disjoint.  

Let $v\in V(\widetilde{F}_k)$ be the vertex with minimum $\sigma(v)$ such that there exists a vertex $u$ with $\sigma(u)<\sigma(v)$ which is disjoint from $v$. Therefore, there is a vertex $w$ with $uvw\in E(\widetilde{F}_k)$ (and then $v$ and $w$ are also disjoint). Further, a vertex $u'$, different from $u$, with $u'vw\in E(\widetilde{F}_k)$ exists.
Since $u,u',v,w$ are pairwise contained in some edge of $\widetilde{F}_k$, they are pairwise disjoint. Then, by the minimality of $v$ we have $\sigma(u)<\sigma(v)<\sigma(u'),\sigma(w)$. But then we reach a contradiction that the pair $vw$ is blue in the edge $uvw$ while it is red in the edge $u'vw$.   
\end{poc}

\cref{counterex} follows immediately from the above two claims.
\end{proof}

\section{Layered $3$-graphs}
In this section, we prove Theorem \ref{layered}, that is, $\pi_{\mathrm{co}}(F)=0$ is equivalent to $\pi_{\points}(F)=0$ for every layered $3$-graph $F$. In \cref{sec:uni}, we first introduce several characterizations of $3$-graphs with vanishing uniform Tur\'{a}n density. Then we introduce in~\cref{sec:hb} the notions of half-bipartite graphs and dense graph distributions and prove a lemma (\cref{hb}) which acts as a bridge connecting $\pi_{\points}$ and $\pi_{\mathrm{co}}$. Then in~\cref{sec:layer-vanishing}, we characterize layered 3-graphs with vanishing uniform Tur\'an density (\cref{linked}), and present a way of gluing layered 3-graphs (\cref{graphsum}) to obtain a new one with vanishing uniform Tur\'an density, which is a key ingredient in the induction step when proving~\cref{layered} in~\cref{sec-final}.

\subsection{$3$-graphs with vanishing uniform Tur\'{a}n density}\label{sec:uni}

 Reiher, R{\"o}dl and Schacht \cite{reiher2018hypergraphs} gave the following characterization of $3$-graphs with vanishing uniform Tur\'{a}n density. 

\begin{theorem}[Reiher, R{\"o}dl and Schacht, \cite{reiher2018hypergraphs}]\label{rodl}
For any $3$-graph $F$, the following are equivalent.
\stepcounter{propcounter}
\begin{enumerate}[label = \rm({\bfseries \Alph{propcounter}\arabic{enumi}})]      
    \item\label{0} $\pi_{\points}(F)=0$;
    \item\label{coloring} there is a labeling $\sigma$ of the vertex set $V(F)$ and a $3$-coloring $\phi: \partial F \rightarrow \{\emph {red,blue,green}\}$ such that every edge $uvw\in E(F)$ with $\sigma(u)<\sigma(v)<\sigma(w)$ satisfies
$$\phi(uv) = \emph{red},~\phi(uw) = \emph{blue},~\phi(vw) = \emph{green}.$$
\end{enumerate}
\end{theorem}

In our approach, the following reformulation of~\cref{rodl} using the structure property of link graphs is more helpful for us.
\begin{lemma}[Ding, Liu, Wang and Yang, \cite{unico}] \label{permutation}
For any $3$-graph $F$, the following are equivalent.
\begin{itemize}
    \item $\pi_{\points}(F)>0$; 
    \item for every labeling $\sigma$ of $V(F)$, there is some vertex $v\in V(F)$ such that $L_{F}(v)$  contains a monotone $P_3$.
\end{itemize}
\end{lemma}

The next lemma strengthens \cref{permutation} for $(2,1)$-type 3-graphs, which will play an important role in our proof. 

\begin{lemma}[Ding, Liu, Wang and Yang, \cite{unico}]\label{strengthen}
        Let $F$ be a $(2,1)$-type $3$-graph with bipartition $V(F)=A\cup B$ such that $|e\cap A|=2$ and $|e\cap B|=1$ for every $e\in E(F)$. Then $\pi_{\points}(F)>0$ if and only if for every labeling of  $A$, there exists $u\in B$ such that $L_F(u)$ contains a monotone $P_3$. 
\end{lemma}

\subsection{Half-bipartite graphs in dense graph distributions}\label{sec:hb}
The first notion is a family of bipartite graphs which is relevant for the characterization of 3-graphs with vanishing uniform Tur\'an density.
\begin{defn}
    A graph $G$ on a vertex set with a labeling $\sigma$ is \emph{half-bipartite} if it does not contain a monotone $P_3$. Define the \emph{complete half-bipartite} graph $B_k$ on $2k$ vertices to be the graph on vertices $u_1,v_1,u_2,v_2,\dots,u_k,v_k$, in increasing order of $\sigma$, and with edges being the pairs $u_iv_j$ for all $1\leq i\leq j\leq k$. 
\end{defn}

It is not hard to see that a half-bipartite graph must be bipartite.
We first observe that the complete half-bipartite graph is a \emph{universal} graph for all half-bipartite graphs. 

\begin{proposition}\label{prop:universal-hb}
       Every half-bipartite graph $F$ on $k$ vertices is an ordered subgraph of $B_k$.
\end{proposition}
\begin{proof}
    Let the vertices of $F$ be $w_1,w_2,\dots, w_k$ in increasing order, then for each $i\in[k]$, map $w_i$ to $u_i$ if there exists $j>i$ such that $w_iw_j\in E(F)$, and map it to $v_i$ otherwise.
\end{proof}

Note that by Lemma~\ref{permutation} a characterization of $3$-graphs $F$ with $\pi_{\points}(F)=0$ is that there exists a labeling of $V(F)$ in which the link graph of every vertex is half-bipartite.

Another tool that we will use are graph distributions, which can be thought of as probability distributions on the space of graphs on a fixed set of vertices. In other words, for each graph $G$ on a (finite) vertex set $S$, we assign it a non-negative real number $X_G$, which add up to 1. So we get a graph distribution on the vertex set $S$, denoted by $X$, by setting $\Pr(X=G)=X_G$. Further, if for some $\varepsilon>0$, $\Pr(uv\in E(X))\geq\varepsilon$ holds for every pair of vertices $uv$,  then we say that $X$ is \textit{$\varepsilon$-dense}.

Let $H$ be an $n$-vertex $3$-graph satisfying $\delta_{\mathrm{co}}(H)\geq\varepsilon n$. There is a simple way of obtaining an $\varepsilon$-dense graph distribution $X$ out of $H$ by selecting a vertex $v\in V(H)$ uniformly at random, and then taking $X=L_H(v)$.

So far, we have established a link between $\pi_{\points}$ and half-bipartite graphs, and between $\delta_{\mathrm{co}}$ (and thus $\pi_{\mathrm{co}}$) and $\varepsilon$-dense graph distributions. The following lemma connects them all.

\begin{lemma}\label{halfbip}
    For every positive integer $k$ and real number $\varepsilon>0$, there exist a positive integer $n_0$ and a real number $\delta>0$ for which the following hold. For every $\varepsilon$-dense graph distribution $X$ on vertex set $[n]$ with $n\geq n_0$, there exists a set $S$ of $2k$ vertices such that, with probability at least $\delta$, $X$ contains the complete half-bipartite graph $B_k$ in the natural order of numbers on $S$ as a subgraph.
\end{lemma}
 
\begin{proof}
    We will prove this statement by induction on $k$. For $k=1$ the statement is trivial, since $B_1$ is a single edge. Let $B'_k=B_k-u_k$. We prove the induction step in two parts. First we show that if the statement holds for $B_{k-1}$ then it holds for $B'_k$, then we show that if it holds for $B'_k$ then it holds for $B_k$ as well.

    Fix $\varepsilon>0$. Suppose that the statement holds for $k-1$, and let $\delta'$ and $n'_0$ be the corresponding values of $\delta$ and $n_0$. Let $X$ be an $\varepsilon$-dense graph distribution on $[n_1]$, where we will choose the value of $n_1$ later. We construct an auxiliary complete $(2k-2)$-graph $H_1$ on $[n_1]$ as follows. For every set $S$ of $2k-2$ vertices, if the complete half-bipartite graph $B_{k-1}$ on $S$ appears in $X$ with probability at least $\delta'$, then color the edge $S$ of $H_1$ in red, otherwise color $S$ in blue.

    By the induction hypothesis, the coloring of $H_1$ cannot contain a blue clique of size $n'_0$. Therefore, by Ramsey's theorem, there is an integer $n'_1$ such that taking $n_1\geq n'_1$ concludes that $H_1$ contains a red clique of size $2k-3+t$ for $t=\lfloor1/\delta'\rfloor+1$. Denote the vertices in this red clique as $u_1, v_1, u_2, v_2, \dots, u_{k-1}, w_1, w_2, \dots, w_t$, in increasing order. For every $1\leq i\leq t$, the probability that the graph $B_{k-1}$ on vertices $u_1, v_1, u_2, v_2, \dots, u_{k-1}, w_i$ appears in $X$ is at least $\delta'$. These probabilities add up to $\delta' t>1$, meaning that there exist $1\leq i<j\leq t$ such that with probability at least $(\delta't-1)/{t \choose 2}=:\delta_1$ the graph $B_{k-1}$ appears simultaneously on $u_1, v_1, u_2, v_2, \dots, u_{k-1}, w_i$ and $u_1, v_1, u_2, v_2, \dots, u_{k-1}, w_j$. Then on vertices $u_1, v_1, u_2, v_2, \dots, u_{k-1}, w_i, w_j$, a copy of $B'_k$ appears in $X$ with probability at least $\delta_1$. This completes the first part of the induction step.

    The second part of the induction step is similar to the first one. Take $n_2$ large enough, and let $X$ be an $\varepsilon$-dense graph distribution on $[n_2]$. We construct a $2$-coloring of the complete $(2k-1)$-graph $H_2$ on $[n_2]$ by coloring the edge $S$ in red if the graph $B'_k$ on $S$ appears in $X$ with probability at least $\delta_1$, and blue otherwise. Since there cannot be a blue clique of size $n'_1$, by Ramsey's theorem there exists a red clique of size $2k-3+2s$ for $s=\lfloor 1/\delta_1\rfloor+1$. Denote the vertices of this clique as $u_1, v_1, u_2, v_2, \dots, u_{k-2}, v_{k-2}, w_1, z_1, w_2, z_2, \dots, w_s, z_s, v_k$. There exist $1\leq i<j\leq s$ such that, with probability at least $\delta:=(\delta_1 s-1)/{s \choose 2}$, the graph $B'_k$ appears simultaneously on $u_1, v_1, u_2, v_2, \dots, u_{k-2}, v_{k-2}, w_i, z_i, v_k$ and $u_1, v_1, u_2, v_2, \dots, u_{k-2}, v_{k-2}, w_j, z_j, v_k$. Then a copy of $B_k$ appears on the vertices $u_1, v_1, u_2, v_2, \dots, u_{k-2}, v_{k-2}, w_i, z_i, w_j, v_k$ with probability at least $\delta$. This completes the second part of the induction step.
\end{proof}

With an extra application of Ramsey's theorem, we can find a set $S$ of $k$ vertices in any large $\varepsilon$-dense graph distribution such that every half-bipartite graph on $S$ appears with a positive probability.

\begin{lemma}\label{hb}
    For every $\varepsilon>0$ and positive integer $k$, there exist a positive integer $n_0$ and a real number $\delta>0$ with the following property. For every $\varepsilon$-dense graph distribution $X$ on $[n]$ with $n\geq n_0$, there exists a set $S$ of $k$ vertices on which every half-bipartite graph appears in $X$ in the natural order of numbers with probability at least $\delta$.
\end{lemma}

\begin{proof}
    Let $\delta=\delta_{\ref{halfbip}}(k,\varepsilon), n'_0=n_{0\ref{halfbip}}(k,\varepsilon)$ be as in Lemma~\ref{halfbip}.  Suppose the contrary that for every set $S$ of $k$ vertices, there exists a half-bipartite graph $F$ that appears with probability less than $\delta$. On an auxiliary complete $k$-graph on $[n]$, assign the label $F$ to the edge $S$.

    Since there is a bounded number of half-bipartite graphs on $k$ vertices, by Ramsey's theorem, if $n$ is large enough there exists a clique $K$ of size $n'_0$ where all edges receive the same label $F$. By the choice of $n'_0$, there exists a set of $2k$ vertices within this clique for which the complete half-bipartite graph $B_k$ on these vertices appears in $X$ with probability at least $\delta$. But recall~\cref{prop:universal-hb} that every half-bipartite graph on $k$ vertices is a subgraph of $B_k$. Thus there are $k$ vertices in $K$ for which $F$ appears with probability at least $\delta$, contradicting the fact that this set of $k$ vertices received the label $F$.
\end{proof}

\subsection{Layered 3-graphs with vanishing uniform Tur\'{a}n density}\label{sec:layer-vanishing}

Recall that \cref{rodl} gives a characterization of $3$-graphs with vanishing uniform Tur\'{a}n density. But, in order to prove Theorem~\ref{layered}, it turns out to be more useful to get a specific characterization for layered 3-graphs. Let $F$ be a layered $3$-graph with $k$ layers. For all $1\leq i<j\leq k$, we denote by $F_{i,j}$ the induced subhypergraph of $F$ on the union of the $i$-th layer and the $j$-th layer. Further, we say that $\{i,j\}$ is a \emph{linked pair} if there exists an edge in $F_{i,j}$. Observe that by \ref{uniquemax}, if $\{i,j\}$ is a linked pair, then $F_{i,j}$ is a $(2,1)$-type $3$-graph such that each edge has two vertices on the $i$-th layer and one vertex on the $j$-th layer.

\begin{lemma}\label{linked}
    A layered $3$-graph $F$ has $\pi_{\points}(F)=0$ if and only if $\pi_{\points}(F_{i,j})=0$ for all linked pairs $\{i,j\}$.
\end{lemma}

\begin{proof}
Let $F$ be a layered $3$-graph with $k$ layers. Since $F_{i,j}$ is a subhypergraph of $F$, it clearly holds that if $\pi_{\points}(F)=0$ then $\pi_{\points}(F_{i,j})=0$ for any linked pair $\{i,j\}$. Next we assume that $\pi_{\points}(F_{i,j})=0$ for any linked pair $\{i,j\}$ and prove that $\pi_{\points}(F)=0$. 

 For any linked pair $\{i,j\}$ with $i<j$, since $F_{i,j}$ is a $(2,1)$-type 3-graph, by Lemma~\ref{strengthen} there exists a labeling $\sigma_i$ of the $i$-th layer such that no vertex $v$ on the $j$-th layer has a monotone $P_3$ in its link graph $L_{F_{i,j}}(v)$. By \ref{fixtwo}, there cannot be another integer $\ell>i$ such that $\{i,\ell\}$ is also a linked pair. Therefore, the property above gives exactly one labeling of the $i$-th layer. For each of the remaining layers, we fix an arbitrary labeling and denote by $\sigma_j$ the labeling of the $j$-th layer. Now we define a labeling $\sigma$ of $V(F)$ by letting $\sigma=\sum_{i=1}^k\sigma_i$, and show that $L_F(v)$ does not contain a monotone $P_3$ for any $v\in V(F)$ .

Let $v$ be a vertex on the $i$-th layer. Then each component of $L_F(v)$ is contained in at most two layers as otherwise we can find two edges which violate \ref{fixtwo}.  If a component crosses two layers, say the $j$-th layer and the $\ell$-th layer with $j<\ell$, then the component must be bipartite, because by \ref{fixtwo} each edge containing $v$ and a vertex of the $j$-th layer must have the other vertex in the $\ell$-th layer. Moreover, by the definition of $\sigma$, the label of each vertex in the $j$-th layer is smaller than the one of each vertex in the $k$-th layer, so there is no monotone $P_3$ in such component. 

If a (nontrivial) component is contained in one layer, say the $j$-th layer, then we must have that $\{i,j\}$ is a linked pair with $j<i$. Then, because this component of $L_F(v)$ is also a component of $L_{F_{j,i}}(v)$, and $L_{F_{j,i}}(v)$ contains no monotone $P_3$ in the labeling $\sigma_j$, we conclude that the labeling $\sigma$ produces no monotone $P_3$ in this component. Therefore, the link graph of each vertex does not contain a monotone $P_3$, and $\pi_{\points}(F)=0$ follows from Lemma~\ref{strengthen}.
\end{proof}

The purpose of the subsequent definitions and lemmas is to understand the circumstances under which the union of layered graphs with vanishing uniform Tur\'{a}n density preserves the layered structure and the uniform Tur\'{a}n density. This is a crucial piece in the induction step that allows us to prove Theorem~\ref{layered}  .

Let $F$ be a layered 3-graph on $k$ layers under a layered function $f$. The \textit{reduced graph} of $F$, denoted by $F/f$, is the graph on vertex set $[k]$, in which we add a 3-edge $ijk$ if $F$ contains an edge whose vertices are respectively located on the $i$-th, $j$-th and $k$-th layers, and a directed 2-edge $\vec{ij}$ if $F$ has an edge with two vertices on the $i$-th layer and one on the $j$-th layer. In certain sense, $F/f$ can be seen as the result of contracting each layer of $F$ into a single vertex.

One can easily check that if $F_1$ and $F_2$ are two layered $3$-graphs on the same vertex set which share a common layered function $f$ such that $F_1/f=F_2/f$, then $f$ is also a layered function of the union $F_1\cup F_2$ and $(F_1\cup F_2)/f=F_1/f$, simply because the triples of layers containing edges of $F_1\cup F_2$ have not changed. In fact, the same happens if for some layers, rather than identifying the corresponding vertices in $F_1$ and $F_2$, we take the disjoint union instead. 

Let $F_1, F_2, \dots, F_\ell$ be layered $3$-graphs on the same vertex set which have $k$ layers under a common layered function $f$. Given a subset $S\subseteq \{1,2,\dots,k\}$, we define $\left(\bigoplus_{i=1}^\ell F_i\right)/S$ as the \textit{$S$-union} of these $\ell$ layered $3$-graphs, in which for each $j\in S$ and each vertex $v$ on the $j$-th layer, the corresponding copies of the vertex $v$ in these $\ell$ layered $3$-graphs are identified into a single vertex.

\begin{lemma}\label{graphsum}
    Let $F_1, F_2, \dots, F_\ell$ be layered $3$-graphs on the same vertex set and with the same reduced graph under a common layered function $f$. Suppose that $\pi_{\points}(F_i)=0$ for all $1\leq i\leq \ell$, and all of them share a common labeling $\sigma$ satisfying \ref{coloring}. Then for any set $S$ of layers containing no linked pair, the $3$-graph $F=\left(\bigoplus_{i=1}^\ell F_i\right)/S$ has $\pi_{\points}(F)=0$.
\end{lemma}

\begin{proof}
    For every $1\leq i\leq \ell$ and each $v\in V(F)\cap V(F_i)$, we use $v_i$ to denote the corresponding copy of it in $F_i$. Clearly, for each $v\in V(F)$, assigning $f(v_i)$ to it if it comes from $V(F_i)$ produces a layered function of $F$. Therefore, by Lemma~\ref{linked}, we only need to check that $\pi_{\points}(F_{r,s})=0$ for any linked pair $\{r,s\}$. 
    
    The condition in the statement tells us that $\{r,s\}\not\subseteq S$. If neither $r$ nor $s$ is in $S$, then $F_{r,s}$ is just the disjoint union of all the $3$-graphs $(F_i)_{r,s}$, and therefore $\pi_{\points}(F_{r,s})=0$. Next we assume that exactly one of $\{r,s\}$ is in $S$. 
    If $r\in S$, then let $\sigma_r$ be the labeling of the $r$-th layer of $F$ which is induced by $\sigma$. Note that for every vertex $v$ on the $s$-th layer of $F$, there exists a unique $t$ such that $v\in V(F_t)$. Then $L_{F_{r,s}}(v)=L_{(F_t)_{r,s}}(v_t)$, which contains no monotone $P_3$ under $\sigma$, and therefore no monotone $P_3$ under $\sigma_r$. On the other hand, if $s\in S$, then for all $1\leq i\leq \ell$, let $\sigma_{i,r}$ be the labeling of the $r$-th layer of $F_i$ which is induced by $\sigma$, and let $\sigma_r=\sum_{i=1}^{\ell}\sigma_{i,r}$ be the labeling of the $r$-th layer of $F$. For each vertex $v$ on the $s$-th layer of $F$, the link graph $L_{F_{r,s}}(v)$ is the disjoint union of all the link graphs $L_{(F_i)_{r,s}}(v_i)$, each of which contains no monotone $P_3$ under $\sigma$, and therefore no monotone $P_3$ under $\sigma_r$. We conclude that in all three cases $\pi_{\points}(F_{r,s})=0$.
\end{proof}

\subsection{Putting things together}\label{sec-final}
We are almost ready to prove Theorem~\ref{layered}. The last piece of the puzzle comes from finding a way to build $F$ layer by layer. In our case, this will mean that on each step we build a tripartite $3$-graph on top of two existing layers, or a $(2,1)$-type $3$-graph on top of an existing layer. We will use the two following lemmas for these purposes.

\begin{lemma}\label{addtripartite}
    For every $\varepsilon>0$ and integer $t$, there exists an integer $m$ with the following property. If $H$ is a tripartite $3$-graph with vertex partition $U\cup V\cup W$, each with size at least $m$, and for every $u\in U$ and $v\in V$ we have $d(u,v)\geq \varepsilon|W|$, then $K^{(3)}_{t,t,t}$ is a subgraph of $H$.
\end{lemma}

\begin{proof}
    By double counting on pairs $(u,v)$, we see that $H$ has at least $\varepsilon|U||V||W|$ edges. By a simple probabilistic argument, there exist sets $U'\subseteq U$, $V'\subseteq V$, $W'\subseteq W$, each with size exactly $m$, spanning at least $\varepsilon m^3$ edges. This means that the 3-graph on these vertices has density at least $2\varepsilon/9$. The lemma follows from the fact that $\pi(K^{(3)}_{t,t,t})=0$.
\end{proof}

\begin{lemma}\label{add21}
    For every $(2,1)$-type $3$-graph $F$ with $\pi_{\points}(F)=0$ and every $\varepsilon>0$, there exists an integer $m$ with the following property. If $H$ is a $(2,1)$-type $3$-graph with vertex partition $U\cup V$, each with size at least $m$, and for each pair $u_1, u_2\in U$ we have $d_H(u_1,u_2)\geq \varepsilon|V|$, then $F$ is a subgraph of $H$.
\end{lemma}
 
\begin{proof}
    Let $U'\cup V'$ be a corresponding vertex partition of $V(F)$ with $|U'|=k$. Take an ordering of the vertices of $U$, and consider the graph distribution $X$ on $U$, obtained by randomly sampling a vertex $v\in V$ and taking $X=L_H(v)$. Since $d_H(u_1,u_2)\geq \varepsilon|V|$, we have that $X$ is $\varepsilon$-dense. Therefore, by Lemma~\ref{hb}, there exists $\delta=\delta_{\ref{hb}}(k,\varepsilon)$ such that, if $m$ is large enough, there exists a set $S\subseteq U$ of size $k$ for which every half-bipartite graph on $S$ appears in $X$ with probability at least $\delta$.

    Since $\pi_{\points}(F)=0$, there exists a labeling of $U'$ in which no vertex $v$ of $V'$ contains a monotone $P_3$ in its link graph, in other words, $L_F(v)$ is a half-bipartite graph on $U'$ for each $v\in V'$. Now we identify the vertices of $U'$ with the vertices of $S$, in the same order. If $|V|\ge m\geq |V'|/\delta$, then for each $v\in V'$, there exist at least $|V'|$ vertices of $V$ each of which has $L_F(v)$ as a subgraph of its link graph. Thus we can select the image of the vertices of $V'$ one by one from $V$, ensuring that we do not select the same vertex twice, to complete the copy of $F$ in $H$.
\end{proof}

We are now ready to prove~\cref{layered}.
\begin{proof}[Proof of Theorem~\ref{layered}]
Let $F$ be a layered 3-graph on $k$ layers with $\pi_{\points}(F)=0$. We proceed our proof by induction on $k$ to show that $\pi_{\co}(F)=0$. If $k=1$, then by \ref{uniquemax} the $3$-graph $F$ has no edge, and so $\pi_{\mathrm{co}}(F)=0$ holds trivially. Now we assume that $k\geq 2$, and the induction hypothesis is that all layered 3-graphs $F'$ on $k-1$ layers with $\pi_{\points}(F')=0$ satisfy $\pi_{\co}(F')=0$. Given $\varepsilon>0$, we choose 
$$1/n\ll 1/m\ll \varepsilon,1/|V(F)|.$$
Let $H$ be an $n$-vertex $3$-graph with $\delta_{\co}(H)\geq\varepsilon n$. Our goal is to show that $H$ contains $F$ as a subgraph.

As $\pi_{\points}(F)=0$, we can fix a labeling $\sigma$ of $F$ as in Lemma~\ref{permutation} so that the link graph of every vertex does not contain a monotone $P_3$. For $1\leq i\leq k$, denote by $L_i$ the $i$-th layer of $F$. Let $\widetilde{F}$ be the layered $3$-graph on $k-1$ layers obtained from $F$ by deleting all vertices in $L_k$. Clearly, $\pi_{\points}(\widetilde{F})=0$, and so $\pi_{\co}(\widetilde{F})=0$ follows from our induction hypothesis. If no edge in $F$ intersects $L_k$, then $\pi_{\co}(F)=\pi_{\co}(\widetilde{F})=0$. So we assume that at least one edge intersects $L_k$, and then by \ref{uniquemax}, each such edge contains a unique vertex in $L_k$. Moreover, by \ref{fixmax}, the layers $L_i$ and $L_j$ containing the other two vertices in each such edge are the same. We consider the two following cases to finish our proof. 

\vspace{0.2cm}

\noindent {\bf Case 1: \bm{$i=j$}.} 
Note that $\{i,k\}$ is a linked pair and $\pi_{\points}(F_{i,k})=0$. Setting $F_{i,k}$ and $\varepsilon/2$ into Lemma~\ref{add21} produces a value of $m$. Suppose that $L_i=\{v_1,v_2,\dots,v_t\}$ with $\sigma(v_1)<\sigma(v_2)<\cdots<\sigma(v_t)$. For every $t$-subset $S$ of $[m]$ with $S=\{s_1,s_2,\dots,s_t\}$ and $s_1<s_2<\cdots<s_t$, we construct a $3$-graph $\widetilde{F}_S$ on $k-1$ layers from $\widetilde{F}$ by replacing $L_i$ with $[m]$ such that $s_{\ell}$ is a copy of $v_{\ell}$ for every $1\leq\ell\leq t$ and $r$ is an isolated vertex for each $r\in[m]\setminus S$.

Consider the $3$-graph $F'=\left(\bigoplus_{S\in{[m] \choose t}} \widetilde{F}_S\right)/\{i\}$. Observe that by Lemma~\ref{graphsum}, $F'$ is a layered $3$-graph on $k-1$ layers with $\pi_{\points}(F')=0$. Thus $\pi_{\mathrm{co}}(F')=0$ follows from our induction hypothesis, implying that $H$ contains a copy of $F'$ if $n$ is large enough. Now let $U$ be the vertices of $H$ corresponding to the $i$-th layer of $F'$ (which, remember, is $[m]$), and let $V$ be the set of vertices of $H$ not in the copy of $F'$. If $n$ is large enough, then we have both $|V|\geq m$ and every pair of vertices in $U$ has codegree in $V$ at least $\varepsilon n/2$. By the choice of $m$, there exists a copy of $F_{i,k}$ between $U$ and $V$. If $S$ is the set of vertices of $[m]$ corresponding to the vertices of $U$ in the copy of $F_{i,k}$, then the copy of $F_S$ together with the copy of $F_{i,k}$ form a copy of $F$ in $H$, as we wanted.
\vspace{0.2cm}

\noindent {\bf Case 2: \bm{$i\neq j$}.} Suppose that $L_i=\{u_1,u_2,\dots,u_{t_1}\}$ with $\sigma(u_1)<\sigma(u_2)<\cdots<\sigma(u_{t_1})$ and $L_j=\{w_1,w_2,\dots,w_{t_2}\}$ with $\sigma(w_1)<\sigma(w_2)<\cdots<\sigma(w_{t_2})$. Setting $\varepsilon/2$ and $t=\max\{|L_i|,|L_j|,|L_k|\}$ into Lemma~\ref{addtripartite} produces a value of $m$. Let $\mathcal{X}=[m]$ and $\mathcal{Y}=[2m]\setminus[m]$. For every $t_1$-subset $X=\{x_1,x_2,\dots,x_{t_1}\}\subseteq\mathcal{X}$ and every $t_2$-subset $Y=\{y_1,y_2,\dots,y_{t_2}\}\subseteq\mathcal{Y}$ with $x_1<x_2<\cdots<x_{t_1}$ and $y_1<y_2<\cdots<y_{t_2}$, we construct a $3$-graph $\widetilde{F}_{X,Y}$ on $k-1$ layers from $\widetilde{F}$ by replacing $L_i$ with $\mathcal{X}$ and replacing $L_j$ with $\mathcal{Y}$ such that $x_{r}$ is a copy of $u_{r}$ for every $1\leq r\leq t_1$, $y_{s}$ is a copy of $w_{s}$ for every $1\leq s\leq t_2$ and $\ell$ is an isolated vertex for each $\ell\in[2m]\setminus (X\cup Y)$.

Consider the $3$-graph $F'=\left(\bigoplus_{(X,Y)\in{\mathcal{X} \choose t_1}\times{\mathcal{Y} \choose t_2}} \widetilde{F}_{X,Y}\right)/\{i,j\}$, which is a $3$-graph on $k-1$ layers. Note also that by~\ref{fixtwo}, $\{i,j\}$ is not a linked pair, and hence $\pi_{\points}(F')=0$ by Lemma~\ref{graphsum}. By our induction hypothesis, we have $\pi_{\mathrm{co}}(F')=0$, meaning that $H$ contains a copy of $F'$ if $n$ is large enough. Let $U$ and $V$ be the sets of vertices on the $i$-th and $j$-th layers of this copy of $F'$, and let $W$ be the set of vertices of $H$ outside of this copy of $F'$. If $n$ is large enough, then we have that $|W|\geq m$ and that for all pairs $uv\in U\times V$, the codegree of $uv$ in $W$ is at least $\varepsilon n/2$. By the choice of $m$, there exists a copy of $K^{(3)}_{t,t,t}$ between $U$, $V$ and $W$, and therefore a copy of $K^{(3)}_{t_1,t_2,|L_k|}$ with $t_1,t_2,|L_k|$ vertices from $U,V,W$, respectively. If $X$ and $Y$ are the sets of vertices of $\mathcal{X}$ and $\mathcal{Y}$ corresponding to the vertices of $U$ and $V$ in the copy of $K^{(3)}_{t_1,t_2,|L_k|}$, then the copy of $F_{X,Y}$ together with the copy of $K^{(3)}_{t_1,t_2,|L_k|}$ form a copy of $F$ in $H$, as we wanted.

We have shown that in all cases $H$ contains $F$ as a subgraph, and so $\pi_{\mathrm{co}}(F)=0$ as desired.
\end{proof}

\section{Is the layered structure necessary?}
In order to prove Conjecture~\ref{Conj1}, one would need to prove that all hypergraphs $F$ which are not layered satisfy $\pi_{\mathrm{co}}(F)>0$. It seems, unlike the usual Tur\'an density $\pi$ and the uniform Tur\'an density $\pi_{\points}$, there are no simple explicit constructions with linear codegree that work for \emph{all} non-layered hypergraphs. We give an example to illustrates the difficulty in trying to find such simple explicit constructions. This example comes from gluing together two layered 3-graphs. While both layered 3-graphs are very simple and have codegree Tur\'an density zero, the resulting graph $F$ is not layered and the simplest constructions showing $\pi_{\co}(F)>0$ we can find is already relatively complex.


\subsection{ A specific example}

Let $F_1$ be the hypergraph on vertices $\{a,b,c,d,e\}$, with edges $\{abc, abd, cde\}$. Let $F_2$ be the hypergraph on vertices $\{e,f,g,h,i,j,k,\ell\}$ with edges $\{fgh, fgi, hij, hik, jk\ell, eh\ell\}$. Let $F=F_1\cup F_2$.

\begin{figure}[h]
\centering
\includegraphics[width=0.3\textwidth]{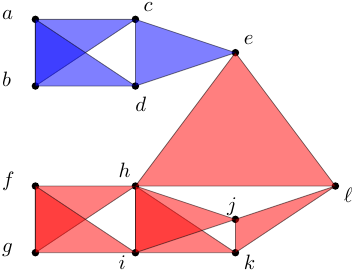}
\caption{The hypergraph $F$, with $F_1$ in blue and $F_2$ in red.}
\end{figure}

Observe that $F_1$ and $F_2$ are both layered and $\pi_{\points}(F_1)=\pi_{\points}(F_2)=0$, so by Theorem~\ref{layered} we have $\pi_{\mathrm{co}}(F_1)=\pi_{\mathrm{co}}(F_2)=0$. On the other hand, one can check that $F$ is not layered.

We now give a construction showing $\pi_{\mathrm{co}}(F)\geq 1/12$, and in particular that the codegree density of $F$ is positive. Consider the following hypergraph $H$ with vertex set is partitioned into twelve equal parts $\mathcal{P}=\{A, B, C, D, E, F, G, H, I, J, K, L\}$, and an edge is placed if it belongs to one of the triples in 
\begin{align*}T=\{AAB&, &ACI&, &ADG&, &AEE&, &AFF&, &AHJ&, &AKL&, &BBC&, &BDJ&, &BEH&,\\ BFK&, &BGG&, &BIL&, &CCD&, &CEF&, &CGK&, &CHH&, &CJL&, &DDE&, &DFL&,\\ DHK&, &DII&, &EGI&, &EJJ&, &EKK&, &ELL&, &FGJ&, &FHI&, &GHL&, &IJK&\}.\end{align*}

A relevant property of $T$ is that every pair of parts $XY$, including those with $X=Y$, is contained in exactly one triple in $T$. Because of this, any pair of edges $uvw$ and $uvw'$ in $H$ which have two common vertices must satisfy that $w$ and $w'$ belong to the same part in $\mathcal{P}$. Observe also that in all triples in $T$ of the form $XXY$ we have $Y\in\{A, B, C, D, E\}$.

In every copy of $F_1$ in $H$, because of the pair of edges $abc$ and $abd$, the vertices $c$ and $d$ belong to the same part of $\mathcal{P}$, which forces $e\in \{A, B, C, D, E\}$. Meanwhile, in every copy of $F_2$ in $H$, because of the pair of edges $fgh$ and $fgi$ we have that $h$ and $i$ belong to the same part of $\mathcal{P}$. This part determines the partition class of $j$ and $k$ through the edges $hij$ and $hik$, which in turn determine the class of $\ell$ through $jk\ell$, which finally determines the class of $e$ through $eh\ell$. Checking all twelve classes for $h$ and $i$, we find that for none of them the resulting vertex $e$ is in $\{A, B, C, D, E\}$, meaning that there is no copy of $F$ in $H$.

We remark that vertex transitive hypergraphs with linear minimum codegree will contain $F$: there will be copies of $F_1$ and $F_2$, of which we can find isomorphisms in which the vertex $e$ coincides.

\subsection{Equivalence of Conjectures \ref{Conj1} and \ref{Conj2}}
In this section, we prove \cref{thm:C1-C2-same}, that is, Conjectures \ref{Conj1} and \ref{Conj2} are equivalent. By Theorem \ref{layered}, to prove the equivalence, it suffices to show that any $3$-graph $F$ with $\pi_{\co}(F)=0$ is layered assuming that this holds for linear $3$-graphs. For this purpose, we will use the following operation to transform a $3$-graph into a linear $3$-graph while preserving the (non)layered structure.

For any $3$-graph $F$ and any vertex $v\in V(F)$, let $F_v$ be the $3$-graph obtained by the following operation.
    \begin{itemize}
        \item Delete vertex $v$, and for each $uw\in L_F(v)$, add a new vertex $v_{uw}$ and a new edge $uwv_{uw}$.  
        \item Add three new vertices $x_v,y_v,z_v$, and then for each $uw\in L_F(v)$, add three more new vertices $x_{uw},y_{uw},z_{uw}$ and the following six new edges 
        $$x_vv_{uw}x_{uw},x_vy_{uw}z_{uw},y_vv_{uw}y_{uw},y_vx_{uw}z_{uw},z_vv_{uw}z_{uw},z_vx_{uw}y_{uw},$$
        which is a Fano plane on $\{x_v,y_v,z_v,v_{uw},x_{uw},y_{uw},z_{uw}\}$ with one edge $x_vy_vz_v$ removed and is denoted by $\mathbf{F}_{uw}$. 
    \end{itemize}
Denote by $\mathcal{L}_v=\{x_v,y_v,z_v\}\cup\{v_{uw},x_{uw},y_{uw},z_{uw}:uw\in L_F(v)\}$ the collection of all these new vertices. Note that for each $u\in\mathcal{L}_v$, its link graph is a matching. We say that the vertex $v$ is \textit{linearized}. 

Given a $3$-graph $F$ and a function $f:V(F)\rightarrow \mathbb{N}$, we say that $f$ is a \textit{semi-layered function} of $F$ if it satisfies conditions \ref{uniquemax} and \ref{fixmax}, and define the \textit{cardinality} of $f$ to be the size of its range. Furthermore, a semi-layered function is \textit{minimum} if it has the minimum cardinality over all semi-layered functions. For the sake of convenience, for any edge $e=uvw\in E(F)$, we use $f(e)$ to denote the multiset $\{f(u),f(v),f(w)\}$.

\begin{proposition}\label{func}
Every minimum semi-layered function is a layered function.  
\end{proposition}

\begin{proof}
Let $F$ be a $3$-graph and $f$ be a minimum semi-layered function of $F$. Suppose on the contrary that $f$ is not a layered function of $F$. Then there are two edges $e,e'\in E(F)$ satisfying that $|f(e)\cap f(e')|=2$. Therefore, $\max f(e)\neq\max f(e')$ holds as otherwise $f(e)=f(e')$ would follow from condition \ref{fixmax}. Let $p=\max f(e)$, $t=\max f(e')$ and $q$ be the unique element in $f(e')\setminus f(e)$. Without loss of generality, we may assume that $p>t\geq q$. For each $v\in V(F)$, let 

\begin{equation}
 g(v)=
\begin{cases}
f(v),    &  \text{if $f(v)\neq p$;}\\
q,    &  \text{if $f(v)= p$.}\nonumber
\end{cases}   
\end{equation}
Note that $g(e)=g(e')=f(e')$, and it is not hard to see that

($\star$)~either $\max g(\hat{e})=\max f(\hat{e})$ or $\max g(\hat{e})=t$ for any edge $\hat{e}$ with $p\in f(\hat{e})$. 

We show that $g$ is also a semi-layered function of $F$, which leads to a contradiction as the cardinality of $g$ is smaller than the one of $f$.

Let $e_0$ be an edge in $E(F)$ with $\max f(e_0)=\ell$. If $\ell>p$, then $g(e_0)$ has a unique maximum element because $p>q$ and $f$ satisfies \ref{uniquemax}. If $\ell<p$, then $g(e_0)=f(e_0)$ as $p\notin f(e_0)$, implying that $g(e_0)$ also has a unique maximum element. If $\ell=p$, then $f(e_0)=f(e)$ by condition \ref{fixmax}. Therefore, $g(e_0)=g(e)=g(e')=f(e')$, which implies that $g(e_0)$ has a unique maximum element as $f(e')$ does. So $g$ satisfies condition \ref{uniquemax}. 

Let $e_1,e_2$ be two edges with $\max g(e_1)=\max g(e_2)=\ell$. If $\ell\neq t$, then by~($\star$) we must have $\max f(e_1)=\max f(e_2)=\ell$, implying that $f(e_1)=f(e_2)$. Therefore, $g(e_1)=g(e_2)$ follows. If $\ell=t$, then for each $i\in\{1,2\}$, either $\max f(e_i)=t$ or $\max f(e_i)=p$ , which implies that either $f(e_i)=f(e)$ or $f(e_i)=f(e')$. Therefore, $g(e_1)=g(e_2)$ follows from that $g(e_i)\in\{g(e),g(e')\}$ for each $i\in\{1,2\}$ and $g(e)=g(e')$. So $g$ satisfies condition \ref{fixmax}.
\end{proof}

A straightforward corollary of \cref{func} is that a $3$-graph $F$ is layered if and only if $F$ has a semi-layered function. Now we show that linearizing a vertex preserves the (non)layered structure.

\begin{proposition}\label{linear}
 For any $3$-graph $F$ and $v\in V(F)$, $F$ is layered if and only if $F_v$ is layered.   
\end{proposition}

\begin{proof} 
 Suppose that $F$ is layered, and let $f$ be a layered function of $F$. Now we define a function $g:V(F_v)\rightarrow \mathbb{N}$ by setting

\begin{equation}
 g(u)=
\begin{cases}
f(u),    &  \text{if $u\in V(F_v)\setminus\mathcal{L}_v$;}\\
f(v),    &  \text{if $u\in \mathcal{L}_v\setminus\{x_v,y_v,z_v\}$;}\\
N,       &  \text{if $u\in \{x_v,y_v,z_v\}$,}\nonumber
\end{cases}   
\end{equation}
\noindent where $N=\max \{f(u): u\in V(F)\}+1$. 
Next we prove that $F_v$ is layered by showing that $g$ is a semi-layered function of $F_v$.

One can easily verify that $g$ satisfies condition \ref{uniquemax}. Now we show that $g$ also satisfies condition \ref{fixmax}. Let $e_1,e_2\in E(F_v)$ be any two edges with $\max g(e_1)=\max g(e_2)=M$. If $M=N$, then $e_1,e_2\subseteq\mathcal{L}_v$ and $g(e_1)=g(e_2)=\{f(v),f(v),N\}$. If $M\neq N$, then for any $i\in\{1,2\}$, either $e_i\cap\mathcal{L}_v=\varnothing$ or $e_i=uwv_{uw}$ for some $uw\in L_F(v)$. 
By the definition of $g$, it follows that either $g(e_i)=f(e_i)$ or $g(e_i)=f(uwv)$ for some $uw\in L_F(v)$, implying that $g(e_1)=g(e_2)$ as $f$ satisfies condition \ref{fixmax}. Therefore, $g$ satisfies condition \ref{fixmax}.

Reversely, suppose that $F_v$ is layered and $g:V(F_v)\rightarrow \mathbb{N}$ is a layered function of $F_v$. For each $uw\in L_F(v)$, since $\mathbf{F}_{uw}$ is a Fano plane with one edge removed, it is not hard to verify that the layered function $g$ must satisfy that 
$$g(x_v)=g(y_v)=g(z_v)>g(v_{uw})=g(x_{uw})=g(y_{uw})=g(z_{uw}).$$
Furthermore, by condition \ref{fixmax}, we know that $g(v_1)=g(v_2)$ for any two vertices $v_1,v_2\in\mathcal{L}_v\setminus\{x_v,y_v,z_v\}$. 
Let $\hat{v}$ be a vertex in $\mathcal{L}_v\setminus\{x_v,y_v,z_v\}$. Define a new function $f:V(F)\rightarrow \mathbb{N}$ by setting
 \begin{equation}
 f(u)=
\begin{cases}
g(\hat{v}),    &  \text{if $u=v$;}\\
g(u),          &  \text{if $u\in V(F)\setminus \{v\}$.}\nonumber
\end{cases}   
\end{equation}
Note that for any edge $uwv\in E(F)$, 
$$f(uwv)=\{g(u),g(w),g(\hat{v})\}=\{g(u),g(w),g(v_{uw})\}=g(uwv_{uw}).$$
Combining with that $f(e)=g(e)$ for any $e\in E(F)$ not containing $v$, it clearly holds that $f$ is a semi-layered function of $F$, and thus $F$ is layered. 
\end{proof}

The following proposition shows that linearizing a vertex of any $3$-graph will not increase its codegree Tur\'{a}n density.

\begin{proposition}\label{lintur}
For any $3$-graph $F$ and $v\in V(F)$, $\pi_{\co}(F_v)\leq\pi_{\co}(F)$.
\end{proposition}
\begin{proof}
Choose parameters satisfying $1/n\ll1/t\ll\varepsilon,1/|V(F)|$. Let $F(v,t)$ be the $3$-graph obtained from $F$ by blowing up the vertex $v$ into an independent set of size $t$, then $\pi_{\co}(F(v,t))=\pi_{\co}(F)$. Let $H$ be any $n$-vertex $3$-graph with $\delta_{\co}(H)\geq(\pi_{\co}(F)+\varepsilon)n$. Then $F(v,t)$ is a subgraph of $H$. Let $U_1$ be the vertex set corresponding to the blow-up of $v$ and $U_2=V(H)\setminus V(F(v,t))$. Let $H'$ be the $(2,1)$-type subgraph of $H$ with $V(H')=U_1\cup U_2$ and $E(H')=\{e\in H:|e\cap U_1|=2, |e\cap U_2|=1\}$. Therefore, $d_{H'}(u_1,u_2)\geq(\pi_{\co}(F)+\varepsilon/2)n$ for any two vertices $u_1,u_2\in U_1$. Further, note that $F_v[\mathcal{L}_v]$ is a $(2,1)$-type linear $3$-graph with partition $\{v_{uw},x_{uw},y_{uw},z_{uw}:uw\in L_F(v)\}\cup \{x_v,y_v,z_v\}$. We conclude that $F_v[\mathcal{L}_v]$ is a subgraph of $H'$ by \cref{add21}. Hence, $F_v$ is a subgraph of $H$, implying that $\pi_{\co}(F_v)\leq\pi_{\co}(F)$. 
\end{proof}

\begin{proof}[Proof of~\cref{thm:C1-C2-same}]
Note that, by Theorem \ref{layered}, we only need to show that any $3$-graph $F$ with $\pi_{\co}(F)=0$ is layered assuming that this is true for all linear $3$-graphs.

Let $F$ be a $3$-graph with $\pi_{\co}(F)=0$, and $\mathcal{L}(F)$ be a linear $3$-graph obtained from $F$ by linearizing, one by one, all vertices whose link graphs are not matchings. Then it follows from \cref{lintur} that $\pi_{\co}(\mathcal{L}(F))=0$. Suppose that Conjecture \ref{Conj2} is true, then $\mathcal{L}(F)$ is layered, which implies that $F$ is also layered by \cref{linear}. 
\end{proof}

\section{Concluding Remarks}

In this paper, we studied \cref{chac} for $3$-graphs and proved that any layered $3$-graph $F$ with $\pi_{\points}(F)=0$ has $\pi_{\co}(F)=0$. On the other hand, we proved that $\pi_{\points}(F)=0$ is a necessary condition and reduced the problem determining whether the layered structure is necessary to the linear $3$-graph case. One explanation for the difficulty of \cref{chac} is that the codegree Tur\'{a}n density can be arbitrarily close to zero. Towards~\cref{Conj2}, we wonder if zero is also an accumulation point for the codegree Tur\'an density for linear 3-graphs.

\begin{question}
For any $\varepsilon>0$, is there a linear $3$-graph $F$ with $0<\pi_{\co}(F)<\varepsilon$?  
\end{question}

Similar as codegree Tur\'{a}n density, one can define the $s$-degree Tur\'{a}n density, which was first mentioned by Keevash \cite{keevash2009hypergraph} and formally introduced by Lo and Markstr\"{o}m \cite{l-degree}. Let $H$ be an $n$-vertex $k$-graph. For $1\leq s\leq k-1$, the \textit{minimum $s$-degree}, denoted by $\delta_{s}(H)$, is the minimum of $d_H(S)$ over all $s$-subsets $S$ of $V(H)$. Given a family $\mathcal{F}$ of $k$-graphs, 
the \textit{$s$-degree Tur\'{a}n number} $\ex_{s}^k(n,\mathcal{F})$ is the largest $\delta_{s}(H)$ over all $n$-vertex $k$-graphs containing none of the members in $\mathcal{F}$. Similarly, the \textit{$s$-degree Tur\'{a}n density} of $\mathcal{F}$ is defined to be 
$$\pi_{s}^k(\mathcal{F})=\lim\limits_{n\to \infty}\frac{\ex_{s}^k(n,\mathcal{F})}{\binom{n}{k-s}}.$$
Lo and Markstr\"{o}m \cite{l-degree} showed that this limit always exists and $\pi_{s}^k(F)$ also possesses the supersaturation property for any $k$-graph $F$.

Let $\widetilde{\Pi}_{s}^k=\{\pi_{s}^k(\mathcal{F}) : \mathcal{F}~\text{is a family of $k$-graphs}\}$ and $\Pi_{s}^k=\{\pi_{s}^k(F) : F~\text{is a $k$-graph}\}$. Mubayi and Zhao \cite{mubayi2007co} proved that $\widetilde{\Pi}_{s}^k$ is dense in $[0,1)$ for $k\geq 3$ and $s=k-1$. Lo and Markstr\"{o}m \cite{l-degree} later extended this result to all $2\leq s\leq k-1$ (note that $\pi_{1}^k(\mathcal{F})=\pi(\mathcal{F})$, and therefore $\widetilde{\Pi}_{1}^k$ is not dense in $[0,1)$). But the same question for $\Pi_{s}^k$ still remains widely open. 

\begin{question}\label{dense}
For $k\geq 3$ and $2\leq s\leq k-1$, is $\Pi_{s}^k$ dense in $[0,1)$?  
\end{question}

As a positive evidence, Piga and Sch\"{u}lke \cite{piga2023hypergraphs} proved that zero is an accumulation point of $\Pi^k_{k-1}$ for $k\geq 3$. So it would be very interesting to figure out whether zero is also an accumulation point of $\Pi^k_{s}$ with $2\leq s\leq k-2$. 

\begin{question}\label{zero}
For $k\geq 4$ and $2\leq s\leq k-2$, is zero an accumulation point of $\Pi_{s}^k$? 
\end{question}

Generalizing \cref{chac}, one can consider the characterization problem for all $2\leq  s\leq k-1$, and it would be interesting if our results can be extended to these general cases. 

\begin{problem}
 For $k\geq 3$ and $2\leq s\leq k-1$, characterize all $k$-graphs $F$ with $\pi_{s}^k(F)=0$.    
\end{problem}

\medskip

\bibliographystyle{abbrv}
\bibliography{references.bib}
\end{document}